\documentclass{daj}
\usepackage{amsthm,amsmath,amssymb,amsfonts,nicefrac}

\dajAUTHORdetails{%
  title = {Effective Bounds for Restricted $3$-Arithmetic Progressions in $\mathbb{F}_p^n$}, 
  author = {Amey Bhangale, Subhash Khot, and Dor Minzer},
  plaintextauthor = {Amey Bhangale, Subhash Khot, and Dor Minzer},
  keywords = {analysis of boolean functions, extremal combinatorics, density increment.},
}   

\dajEDITORdetails{%
   year={2024},
   number={16},
   received={25 August 2023},   
   published={20 December 2024},  
   doi={10.19086/da.125858},       
}   

\newtheorem{thm}{Theorem}[section]

\newtheorem{lemma}[thm]{Lemma}

\newtheorem{claim}[thm]{Claim}

\newtheorem{definition}[thm]{Definition}
\newtheorem{remark}[thm]{Remark}

\newtheorem{fact}[thm]{Fact}

\newcommand\E{\mathop{\mathbb{E}}}
\newcommand\card[1]{\left| {#1} \right|}
\newcommand\sett[2]{\left\{ \left. #1 \;\right\vert #2 \right\}}

\newcommand\set[1]{{\left\{ #1 \right\}}}
\newcommand\Prob[2]{{\Pr_{#1}\left[ {#2} \right]}}

\newcommand\Expect[2]{{\mathop{\mathbb{E}}_{#1}\left[ {#2} \right]}}

\newcommand\norm[1]{\| #1 \|}

\newcommand\skipi{{\vskip 10pt}}
\newcommand\inner[2]{\langle{#1},{#2}\rangle}
\newcommand\eps{\varepsilon}

\begin{document}

\begin{frontmatter}[classification=text]

\title{Effective Bounds for Restricted $3$-Arithmetic Progressions in $\mathbb{F}_p^n$} 

\author[pgom]{Amey Bhangale\thanks{Supported by the Hellman Fellowship Award.}}
\author[laci]{Subhash Khot\thanks{Supported by
		the NSF Award CCF-1422159, NSF CCF award 2130816, and the Simons Investigator Award.}}
\author[andy]{Dor Minzer\thanks{Supported by a Sloan Research
		Fellowship, NSF CCF award 2227876 and NSF CAREER award 2239160.}}

\begin{abstract}
For a prime $p$, a restricted arithmetic progression in $\mathbb{F}_p^n$ is a triplet of vectors $x, x+a, x+2a$ in which the common difference $a$ is a non-zero
  element from $\{0,1,2\}^n$. What is the size of the largest $A\subseteq \mathbb{F}_p^n$ that is free of restricted arithmetic progressions?
  We show that the density of any such a set is at most $\frac{C}{(\log\log\log n)^c}$, where $c,C>0$ depend only on $p$,
  giving the first reasonable bounds for the density of such sets. Previously, the best known bound was $O(1/\log^{*} n)$,
  which follows from the density Hales-Jewett theorem.
\end{abstract}
\end{frontmatter}

\section{Introduction}
Roth's theorem~\cite{roth1953certain} is one of the cornerstones of extremal combinatorics,
asserting that a set of integers $A\subseteq [n]$ that does not contain an arithmetic progression of size $3$,
namely a triplet of the form $x, x+a, x+2a$, must have density which vanishes with $n$; in particular, Roth
showed that $\card{A}\leq O\left(\frac{n}{\log\log n}\right)$. Much effort has gone into improving upon this result quantitatively (that is, showing
that such set $A$ must be in fact much smaller than what was proved by Roth), extending it to longer length progressions,
as well as proving variants of it in different settings.
Currently, it is known~\cite{kelley2023strong} that such set of integers $A$ may be of size at most
$e^{-(\log n)^{c}} n$ for an absolute constant $c>0$.
As for longer progressions, Szemer\'{e}di's theorem~\cite{szemeredi1975sets} is a well known
(and very impactful) generalization of this result to arbitrarily long arithmetic progressions,
showing that for every $k$ there is a vanishing function $\alpha_k\colon\mathbb{N}\to(0,1)$ such that a set of integers $A\subseteq [n]$
with no arithmetic progressions of length $k$ has size at most $\alpha_k(n) n$. The quantitative bound established by Szemer\'{e}di's proof
is quite weak, and a reasonable one was only proved later on by Gowers'~\cite{gowers2001new} in his highly influential work introducing uniformity norms.
Other variants of this problem that received considerable attention replace the set of integers by a different set (say primes~\cite{green2008primes}
or finite fields~\cite{meshulam1995subsets}), or further restricting the common difference of the progression to be from a specific set~\cite{prendiville2017quantitative,peluse2019quantitative}.

One may ask similar questions in the finite field model. Indeed, following Roth's theorem, Meshulam has shown~\cite{meshulam1995subsets}
that a subset $A\subseteq \mathbb{F}_p^n$ with no arithmetic progression may have size which is at most $O\left(\frac{p^n}{\sqrt{n}}\right)$.
Like Roth, Meshulam's argument is Fourier analytic and proceeds by a density increment argument, in which he shows that a set $A$ which is
free of arithmetic progressions of size $3$ must have considerably larger density inside some hyperplane. Until recently, the quantitative
bounds achieved in this finite field model and in the integer model have been progressing at roughly the same rate. Somewhat surprisingly,
it turns out that significantly stronger bounds can be proved in the finite field model using the polynomial method~\cite{ellenberg2017large}.
Specifically, it is now known that for all primes $p\in\mathbb{N}$ there is $\eps>0$ such that a set $A\subseteq \mathbb{F}_p^n$ with no arithmetic progressions of length $3$
can have size at most $(p-\eps)^n$.

The primary topic of this paper is a variant of Roth's theorem in the finite field model, in which
the common difference is restricted to be from a specific set. Along these lines, the most restrictive common difference type one can think of
is for a triplet of the form $x, x+a, x+2a$ where $x\in\mathbb{F}_p^n$ and $a\in\{0,1\}^n\setminus\{\vec{0}\}$. This question has been asked by
H{\k{a}}z{\l}a, Holenstein and Mossel~\cite{Mossel} (in the ``counting version'' asking if a set $A$ of density $\alpha>0$ must contain at least $\beta(\alpha)>0$
fraction of the restricted $3$-AP's), and highlighted by Green~\cite{Green} who also asked for effective bounds for the size of $A$ not containing any such restricted
$3$-APs.

This paper studies a similar question, wherein the common difference is a bit less restricted, and is allowed to be from $a\in\{0,1,2\}^n\setminus\{\vec{0}\}$.
Previously, for $p>3$ the best known quantitative bounds on the measure of a restricted $3$-AP set $A$ are quite poor and stand at $\mu(A)\leq O(1/\log^{*}(n))$
(which follow from the density Hales-Jewett theorem~\cite{polymath2012new}).

While the main result of this paper is purely in additive combinatorics, the authors view the result, as well as the techniques as
a part of their study of the approximability of satisfiable constraint satisfaction problems~\cite{BKMcsp1,BKMcsp2,BKMcsp3,BKMcsp4,GHZ}.

\subsection{Main Result}
Our main result is the following theorem:
\begin{thm}\label{thm:main_el}
 For all primes $p$ there are $c>0$ and $C>0$ such that if $A\subseteq \mathbb{F}_p^n$ is a restricted $3$-AP free set, then
 $\mu(A)\leq \frac{C}{(\log\log\log n)^{c}}$.
\end{thm}
\begin{remark}
  A few remarks are in order.
  \begin{enumerate}
    \item It is quite possible that a refinement of our techniques may lead to better bounds, however it seems that our methods will not
    be able to achieve bounds better than $O\left(1/(\log(n))^C\right)$ for a large absolute constant $C>0$. As far as we know, it is consistent with the current state of
    knowledge that any restricted $3$-AP free set may have size at most $(p-\eps)^{n}$, where $\eps=\eps(p)>0$, which is completely out of
    reach of any currently known density increment type approach.
    \item The reason we are only able to deal with common differences from $\{0,1,2\}^n\setminus\{\vec{0}\}$ (as opposed to $\{0,1\}^n\setminus\{\vec{0}\}$) is
     quite technical at nature. It has to do with the strength of a certain stability result we use, and it may be the case that stronger stability results exists
     that will allow arguments along the lines of the current paper to handle the more restricted common differences in $\{0,1\}^n\setminus\{\vec{0}\}$.
  \end{enumerate}
\end{remark}


\subsection{Our Techniques}
In this section, we give a high level overview of our techniques. We begin by introducing the main ``inverse type'' theorem that we use in our proof
that comes from the theoretical computer science community, and then explain how to apply it in our setting (as well as the challenges entailed in it).
\subsubsection{The CSP Stability Result}
Our approach is motivated by a recent stability theorem by the authors~\cite{BKMcsp4}, which was motivated by the theoretical computer science
perspective --- the study of the complexity of satisfiable constraint satisfaction problems. Here, we elaborate on that result without explaining that
motivation, and refer the reader to~\cite{BKMcsp1,BKMcsp2,BKMcsp3,BKMcsp4} for a more thorough discussion on that.

Suppose that $\Sigma,\Gamma$ and $\Phi$ are finite alphabets, and let $\mu$ be a distribution over $\Sigma\times \Gamma\times \Phi$ in which the probability
of each atom is at least $\Omega(1)$. What functions $f\colon \Sigma^n\to[-1,1]$, $g\colon \Gamma^n\to[-1,1]$ and $h\colon \Phi^n\to[-1,1]$ can satisfy
that
\begin{equation}\label{eq:intro1}
\card{\Expect{(x,y,z)\sim \mu^{\otimes n}}{f(x)g(y)h(z)}}\geq \eps,
\end{equation}
where $\eps>0$ is thought of as a small constant, and $n$ as going to infinity? If $f,g$ and $h$ are low-degree functions, then is may certainly be the case;
for example, if all of the alphabets are the same, and $\mu$ has $1-\delta$ of its mass on points of the form $(x,x,x)$, then for functions that have most of
their mass on degrees lower than $1/\delta$, the above expectation is essentially $\E_x[f(x)g(x)h(x)]$, and this can certainly be large in absolute value.
In~\cite{BKMcsp1}, the authors propose that, in a sense, besides low-degree functions, the only other case in which the above expectation may be large
is the case that the support
of $\mu$ admits a non-trivial \emph{Abelian embedding}:
\begin{definition}
  We say a set $P\subseteq \Sigma\times\Gamma\times \Phi$ can be Abelianly embedded if there is an Abelian group $(H,+)$
  and $3$ maps $\sigma\colon \Sigma\to H$, $\gamma\colon \Gamma\to H$ and $\phi\colon \Phi\to H$
  such that $\sigma(x)+\gamma(y)+\phi(z) = 0$ for all $(x,y,z)\in P$. We say an embedding $\sigma, \gamma, \phi$ is non-trivial
  if at least one of these maps is not constant.

  We say $\mu$ can be Abelianly embedded if ${\sf supp}(\mu)$ can be Abelianly embedded.
\end{definition}
With this definition,~\cite{BKMcsp1} hypothesized that if $\mu$ has no non-trivial Abelian embedding, then~\eqref{eq:intro1} can hold
only if $f,g$ and $h$ each have a significant weight on the low-degrees. In fact, they hypothesized that such statement should be true
not only for the $3$-ary case, but also for the $k$-ary case for all $k\in\mathbb{N}$ (where the definition of no Abelian embedding
for $P\subseteq \Sigma_1\times\ldots\times\Sigma_k$ is analogous to the above). They established that this statement is true for $k=3$ for
a special case of measures $\mu$ (via a reduction to a problem in non-Abelian Fourier analysis), and subsequently in~\cite{BKMcsp2} that the statement
is true for $k=3$ for all $\mu$.

Morally, this indeed says that the only contributions to~\eqref{eq:intro1} may come at the presence of Abelian embeddings. However, one would have liked
to say that even if there are embeddings, to achieve~\eqref{eq:intro1} one must design functions that ``use'' these embeddings. This was proved in~\cite{BKMcsp4},
and to state that result we need a definition.
\begin{definition}
  For $i,j\in \{1,2,3\}$, we say that a distribution $\mu$ over $\Sigma_1\times \Sigma_2\times \Sigma_3$ is $\{i,j\}$-connected
  if the bipartite graph $G = (\Sigma_i\cup\Sigma_j, E)$ whose edges are $(a,b)\subseteq \Sigma_i\times \Sigma_j$ if there is $v\in{\sf supp}(\mu)$
  such that $v_i = a$ and $v_j = b$, is connected. We say $\mu$ is pairwise connected if it is $\{i,j\}$-connected for all distinct $i,j\in\{1,2,3\}$.
\end{definition}
With this in mind, the result of~\cite{BKMcsp4} asserts that if $\mu$ is a pairwise connected distribution that
doesn't admit non-trivial Abelian embeddings into $(\mathbb{Z},+)$, and for which~\eqref{eq:intro1} holds,
then there is an $r$ depending only on the alphabet sizes, an Abelian group $(H,+)$ of size at most $r$,
an Abelian embedding $(\sigma,\gamma,\phi)$ of $\mu$ into $(H,+)$,
characters $\chi_1,\ldots,\chi_n\in\hat{H}$ and a function $L\colon \Sigma^n\to\mathbb{C}$ of degree at most $d$ and $2$-norm at most $1$ such that
\begin{equation}\label{eq:intro2}
\card{\Expect{x\sim\mu_x^{\otimes n}}{f(x)\overline{\prod\limits_{i=1}^{n}\chi_i(\sigma(x_i))L(x)}}}\geq \delta.
\end{equation}
Here, $d,\delta$ depend only on $\eps$ (and with reasonable quantitative bounds).
In words, this result asserts that if $f$ satisfies~\eqref{eq:intro1}
for some bounded functions $g$ and $h$, then it must be correlated with a product of an ``embedding product function'' and a low-degree function (which
is a combination of the only two ``obvious'' ways~\eqref{eq:intro1} can be achieved). See Theorem~\ref{thm:stab} for a more precise statement.

\subsubsection{Stability for Restricted $3$-AP Sets}
In this section we explain how to apply the above stability result to prove Theorem~\ref{thm:main_el}.

Thinking of the distribution $\mu$ over $\mathbb{F}_p\times \mathbb{F}_p\times \mathbb{F}_p$ of $(x,x+a,x+2a)$ where $x\sim\mathbb{F}_p$ and $a\sim\{0,1,2\}$
are sampled uniformly, one is tempted to look at the count of restricted $3$-AP's in $A\subseteq\mathbb{F}_p^n$, namely
\[
\Expect{(x,y,z)\sim \mu^{\otimes n}}{1_A(x)1_A(y)1_A(z)},
\]
and argue that if $A$ has no restricted $3$-AP's, then only trivial $3$-AP's (those with $a=0$) contribute to the above expectation, hence it is at most
$3^{-n}$. Thus,
\[
\Expect{(x,y,z)\sim \mu^{\otimes n}}{(1_A-\mu(A))(x)1_A(y)1_A(z)}
\leq 3^{-n} - \mu(A)\Expect{(x,y,z)\sim \mu^{\otimes n}}{1_A(y)1_A(z)},
\]
and as it is natural to expect that $\Expect{(x,y,z)\sim \mu^{\otimes n}}{1_A(y)1_A(z)}$ is significant,
one gets that the left side above is large in absolute value.
This is now an inequality as~\eqref{eq:intro1} for $f = 1_A - \mu(A)$ and $g = h = 1_A$, and the stability result above may thus be applied.
Hence, $f$ must be correlated with a function
as in~\eqref{eq:intro2}. With any luck, the only embeddings of $\mu$ are the trivial linear embeddings into $\mathbb{F}_p$ given
as $\sigma(x) = \gamma(x) = \phi(x) = x$, at which point one would morally get that $f$ has a significant Fourier coefficient.
As in Roth's theorem, there is now hope that one can use this information to reduce the question to the same question but over a denser set and thus proceed by
a density increment argument. There are several issues with this simplistic description:
\begin{enumerate}
  \item First, in the stability result~\eqref{eq:intro2} one gets a correlation with a product functions times a low-degree function $L$, and one cannot simply
  ignore $L$. In other words, the information we get is not really about the Fourier coefficients of $f$.
  \item Second, even if we found a large Fourier coefficient, it is not immediately clear how to actually do the density increment argument. In Roth/ Meshulam's
  theorem one simply passes on to one of the hyperplanes defined by the large Fourier coefficient on which $A$ is denser. In our situation, as we have the restriction
  that $a$ must be from $\{0,1,2\}^n$, we must take that into account when reducing ourselves to subspaces on which $A$ is denser.
  \item Third, in principle there may be non-trivial embeddings of $\mu$.
\end{enumerate}
Nevertheless, we show that a density increment approach is feasible, and address each one of these issues separately.
Below, we give a high level overview of how each one of the above issues is handled.

\paragraph{Addressing the first issue.} The first issue is the easiest one to handle,
and we do so via \emph{random restrictions}. By that, we mean that we choose a subset $I\subset [n]$ randomly by including each coordinate $i\in [n]$ with
probability $1/2d$, choose $y\sim \mathbb{F}_p^{\overline{I}}$ uniformly and consider the function $f_{\overline{I}\rightarrow y}\colon\mathbb{F}_p^{I}\to[-1,1]$
defined by $f_{\overline{I}\rightarrow y}(z) = f(x_I = z, x_{\overline{I}} = y)$; here, the point $(x_I = z, x_{\overline{I}} = y)$ is the point in which
coordinates of $I$ are filled according to $z$, and the rest are filled according to $y$. It can be shown that with noticeable probability, if~\eqref{eq:intro2}
holds then after choosing $I$ and $y$ as above, the function $f_{\overline{I}\rightarrow y}$ is correlated with a product function (which is the restriction
of the product function from~\eqref{eq:intro2}), and thereby we have eliminated $L$ altogether. Using such ideas repeatedly, we are able to get rid of low-degree
functions $L$ whenever they are present, but remark that this often comes at the expense of some complications. Henceforth, for the rest of
this proof overview we ignore $L$ in~\eqref{eq:intro2},
and pretend that this theorem gives us a correlation with a product function $P = \prod\limits_{i=1}^{n}(\chi_i\circ \sigma)(x_i)$.

\paragraph{Addressing the second issue.}
For the second issue (ignoring the third issue for now), we show that one can indeed go to a subspace (of not too much smaller dimension) in a way that
preserves the structure of restricted $3$-AP's and at the same time increases the density of $A$.
Indeed, write $P = \prod\limits_{i=1}^{n} (\chi_i\circ\sigma)(x_i)$ and suppose that the functions we multiply amount to characters over $\mathbb{F}_p$, say
$(\chi_i\circ\sigma)(x_i) = e^{\frac{2\pi}{p}\alpha_i x_i {\bf i}}$ for $\alpha_i\in \mathbb{F}_p$.
Then by the Pigeonhole principle one may find a value $\alpha\in \mathbb{F}_p$ such that $\alpha_i = \alpha$ for at least $n/p$ of the coordinates $i$.
Let this set of coordinates be denote by $I$ and remove from it up to $p-1$ coordinates if necessary so that its size is divisible by $p$.
Hence, we may find at least $m = \Omega(n)$ disjoint sets in $I$ of size exactly $p$, say $I_1,\ldots,I_m$,
and we take $v_1,\ldots,v_m$ to be their indicator vectors, respectively. Noting that $v_1,\ldots,v_m$ are linearly independent, we may complete $v_1,\ldots,v_m$ to a basis
by adding vectors $u_1,\ldots,u_{n-m}$. Define the change of basis map $M\colon\mathbb{F}_p^n\to\mathbb{F}_p^n$ as
\[
w = M(x_1,\ldots,x_m,z_1,\ldots,z_{n-m}) = \sum\limits_{i=1}^m x_i v_i + \sum\limits_{i=1}^{n-m} z_i u_i.
\]
We now consider restrictions of $f$ of the form $f_z = (f\circ M)_{[n]\setminus [m]\rightarrow z}$.
First, note that the set of inputs to $f_z$ corresponds to an affine subspace of dimension $m$ in $\mathbb{F}_p^n$ (which is just the image of
$M$ when the $z$-variables are fixed). Also note that these affine subspaces uniformly cover $\mathbb{F}_p^n$, in the sense that
each point in $\mathbb{F}_p^n$ appears in the same number of these affine subspaces.
Furthermore, $f_z$ is restricted $3$-AP free for each $z$. Indeed, any non-zero common difference in
it translates to a common difference of the form $\sum\limits_{i=1}^m a_i v_i$ in the original domain for $a_1,\ldots,a_m\in \{0,1,2\}$.
Observe that it is a common difference from $\{0,1,2\}^n$ for $f$, as the supports of the $v_i$'s are disjoint and each one of them is $\{0,1\}$-valued.

Thus, after restricting $z$, namely, looking
at the affine subspace $\sum\limits_{i=1}^{n-m} z_i u_i + {\sf Span}(v_1,\ldots,v_m)$ and the set $A_z = f_z^{-1}(1)$ in it,
one gets the same type of restricted $3$-AP problem. The point now is that
on each one of these affine subspaces, the value of $P$ is constant depending only on $z$.
Indeed, on an affine subspace the value of $P$ is
\[
\prod\limits_{k=1}^{m}\prod\limits_{i\in {\sf supp}(v_k)} (\chi_i\circ \sigma)(w_i) =
\prod\limits_{k=1}^{m}
\prod\limits_{i\in {\sf supp}(v_k)}e^{\frac{2\pi}{p}\alpha(x_k+H_i(z)){\bf i}}
=
\prod\limits_{k=1}^{m}
e^{\frac{2\pi}{p}\alpha \card{v_k} x_k{\bf i}}
\prod\limits_{i\in {\sf supp}(v_k)}e^{\frac{2\pi}{p}\alpha H_i(z){\bf i}},
\]
which is equal to $\prod\limits_{i\in {\sf supp}(v_k)}e^{\frac{2\pi}{p}\alpha H_i(z){\bf i}}$ as $e^{\frac{2\pi}{p}\alpha \card{v_k} x_k{\bf i}} = 1$ for all $k$
(here we use the fact that the size of the support of $v_k$ is $p$).
Overall, the average of $f_z$ over the choice of $z$ is $\E[f]$, and as $f$ is correlated with $P$ there is variance in the average of $f_z$ on these parts.
Thus we may find $z$ on which $f_z$ has significantly larger density than that of $f$. In the main body of the paper, we refer to such changes of basis as ``specialized change of basis'',
and to the above operation as ``specialized change of basis and restriction of the $z$-part''.

\paragraph{Addressing the third issue.}
While the third issue seems to be more complicated than the second issue, it turns out that it could be resolved using the same idea (but requires
a bit more effort). In a sense, in the above argument we didn't really use the fact that $\chi_i\circ \sigma$ is of any special form; rather, that we could
use the pigeonhole principle so as to find many coordinates $i$ for which this function is the same.
At that point we could ``identify'' these variables (this is effectively what the basis elements $v_i$ do), and make sure that in total the number of $i$'s
we identify in this way is a multiple of $r$, these $\chi_i\circ\sigma$ trivial to a constant.

Thus, effectively we are finding a subspace of $\mathbb{F}_p^n$ on which we can lift restricted $3$-AP's to restricted $3$-AP's in $\mathbb{F}_p^n$,
and on which the function $P(x) = \prod\limits_{i=1}^{n} \chi_i\circ \sigma(x_i)$ becomes constant.
This turns out to be true so long as we can apply the pigeonhole principle, but a bit more care is needed.
Indeed, a closer inspection of the effect of the specialized change of basis and restricting the
$z$-part, one gets that $f_z(x) = f(w)$ where $w_i$ either only depends on $z$ (for $i\in [n]\setminus [m]$), or else $w_i = x_k + H_i(z)$
where $k\in [m]$ is the unique $k$ such that $i\in {\sf supp}(v_k)$ and $H_i(z)$ is some linear function in $z$.

The above argument would work if we were able to ``ignore'' the shift $H_i(z)$. Indeed, if no $H_i(z)$ was present than a similar calculation to
before shows that $P$ is constant on the restrictions of the $z$-part after a suitable choice of partition. To see that, we first choose a set $I$ of $\Omega(n)$
coordinates $i$ on which the characters $\chi_i$ are all the same and equal to some $\chi\in\hat{H}$ (this is equivalent to $\alpha_i = \alpha$
in the previous argument). We then partition $I$ into $I_1,\ldots,I_m$ disjoint sets of size $\card{H}$ each, where $m = \Omega(n)$, and take $v_i$ to
be the vector supported only on coordinates in $I_i$. Thus,
\[
\prod\limits_{k=1}^{m}\prod\limits_{i\in {\sf supp}(v_k)} (\chi_i\circ \sigma)(w_i) =
\prod\limits_{k=1}^{m}
\prod\limits_{i\in {\sf supp}(v_k)}\chi(\sigma(x_k + H_i(z)))
=
\prod\limits_{k=1}^{m}
\prod\limits_{i\in {\sf supp}(v_k)}\chi(\sigma(x_k))
=
\prod\limits_{k=1}^{m}
\chi(\sigma(x_k))^{\card{v_k}},
\]
which is equal to $1$ as $\chi^{\card{H}} \equiv 1$ for all $\chi\in\hat{H}$.

We show that in expectation, for many $k\in [m]$ it holds that $H_i(z) = 0$ for all $i\in {\sf supp}(v_k)$,
so we could focus on these $k$'s and restrict the rest, thereby effectively ignore the shifts $H_i(z)$.

\paragraph{Combining the above ideas.} Our formal proof combines the above ideas together, and we show that one can find a significant density bump for $A$
after a sequence of restrictions and specialized change of basis and restriction of the $z$-part. This facilitates a density increment argument that proceeds
iteratively, and following the argument closely and the quantitative bounds in it, one gets Theorem~\ref{thm:main_el}.

\section{Preliminaries}\label{sec:prelim}
In this section we present several, mostly standard notions from analysis of Boolean functions. We refer the reader to~\cite{Odonell} for a
more systematic presentation.
\subsection{The Efron-Stein Decomposition}
In this paper, we will often deal with function over finite product spaces, namely $f\colon (\Sigma^n,\mu^{\otimes n})\to\mathbb{C}$
where $\mu$ is some probability distribution and $\Sigma$ is a finite alphabet. We think of the space $L_2(\Sigma^n,\mu^{\otimes n})$
as an inner product space endowed with the standard $L_2$ inner product
\[
\inner{f}{g} = \Expect{x\sim\mu^{\otimes n}}{f(x)\overline{g(x)}}.
\]
For a subset of coordinates $S\subseteq [n]$, a function
$f$ is called an $S$-junta if there is $g\colon \Sigma^S\to\mathbb{C}$ such that $f(x) = g(x_S)$ for all $x\in\Sigma^n$. We may
thus define the linear space $V_{\subseteq S}\subseteq L_2(\Sigma^n)$ as the span of all $S$-juntas, and then
$V_{=S} = V_{\subseteq S}\cap\bigcap_{T\subsetneq S} V_{\subseteq T}^{\perp}$. It can be shown that the spaces $V_{=S}$ are mutually orthogonal,
and subsequently that $L_2(\Sigma^n,\mu^{\otimes n})$ is the direct sum of $V_{=S}$. This leads to an orthogonal decomposition of any $f\in L_2(\Sigma^n,\mu^{\otimes n})$,
which is called Efron-Stein decomposition:
\begin{fact}
  Given $f\colon (\Sigma^n,\mu^{\otimes n})\to\mathbb{C}$, there is a unique way to write $f = \sum\limits_{S\subseteq [n]} f^{=S}$,
  where $f^{=S}$ is in $V_{=S}$.
\end{fact}
For $S=\emptyset$, the corresponding term in the Efron-Stein decomposition is the constant function closest to $f$ in $L_2$-norm, which
is $f^{=\emptyset} = \Expect{x\sim \mu}{f(x)}$, which we often abbreviate as $\E[f]$.

The degree decomposition of a function is a coarser decomposition which is sometimes more convenient to work with. For $d=0,\ldots,n$,
we denote $f^{=d} = \sum\limits_{\card{S} = d} f^{=S}$ as well as $f^{\leq d} = \sum\limits_{i=0}^{d} f^{=i}$. We often refer to
$f^{\leq d}$ as the ``degree $d$ part of $f$''.
Next, we define the level $d$ weights of a function $f$.
\begin{definition}
  Given $f\colon (\Sigma^n,\mu^{\otimes n})\to\mathbb{C}$ and $d \in \{0,\ldots,n\}$, the weight of $f$ on level exactly $d$ is defined as
  $W_{=d}[f] = \norm{f^{=d}}_2^2$. The weight of $f$ on level $d$ is defined as $W_{\leq d}[f] = \sum\limits_{i=0}^{d}W_{=i}[f]$, which
  by orthogonality of the functions $f^{=i}$ is equal to $\norm{f^{\leq d}}_2^2$.
\end{definition}

\subsection{Markov Chains}
Given a probability space $(\Sigma,\mu)$, we will often consider Markov chains over it whose stationary distribution is $\mu$.
We often denote these by $\mathrm{T}$, and abusing notations we also think of it as an averaging operator from $L_2(\Sigma,\mu)$
to $L_2(\Sigma,\mu)$ defined as
\[
\mathrm{T} f(x) = \Expect{y\sim \mathrm{T}x}{f(y)}.
\]
We say a Markov chain $\mathrm{T}$ is connected if the graph, whose vertices are $\Sigma$ and the edges are $(a,b)$ if there is
a transition from $a$ to $b$ in $\mathrm{T}$, is connected.

Below, we state a few well known properties of Markov chains that we will need; see for example~\cite{mossel2010gaussian}.
If $\mathrm{T}$ is connected and the probability of each transition is at least $q$, then $\lambda_2(\mathrm{T})\leq 1-\Omega_{q}(1)$.
Given a basic operator $\mathrm{T}$, we often think of its $n$-fold tensor $\mathrm{T}^{\otimes n}$, and once again we also think of it as an averaging operator on $L_2(\Sigma^n,\mu^{\otimes n})$. In the case that $\mu$ is a stationary distribution of $\mathrm{T}$, it is easily shown that the spaces $V_{=S}$ are invariant under $\mathrm{T}^{\otimes n}$, and for each $g\in V_{=S}$ it holds that $\norm{\mathrm{T}^{\otimes n} g}_2\leq \lambda_2(\mathrm{T})^{\card{S}}\norm{g}_2$.

\skipi

The following lemma asserts that if $f,g$ are Boolean functions and $f-\E[f]$ has small weight on the low levels, then the probability that $x\sim\mu^{\otimes n}$
and $y\sim\mathrm{T}^{\otimes n} x$
satisfy that $f(x) = g(y) = 1$ is close to $\E[f]\E[g]$ (which is the probability if $x$ and $y$ were sampled according to $\mu^{\otimes n}$ independently).
In our density increment argument, this lemma will be a way for us to argue that $f-\E[f]$ has considerable weight on the low-levels.
\begin{lemma}\label{lem:small_wt_to_neg}
  Suppose $(\Sigma,\mu)$ is a finite domain and $\mathrm{T}$ is a connected Markov chain with stationary distribution $\mu$, in
  which the probability of each atom is at least $q$. Then for all $\alpha, \beta>0$ there is $d = O_q(\log(1/\alpha\beta))$ such that
  if functions $f,g\colon (\Sigma^n,\mu^{\otimes n})\to\{0,1\}$ have averages $\alpha, \beta$ respectively and
  satisfy that $W_{\leq d}[f-\alpha]\leq \frac{\alpha^{2}\beta^2}{100}$, then
  $\inner{f}{\mathrm{T}^{\otimes n} g}\geq \frac{4}{5}\alpha\beta$.
\end{lemma}
\begin{proof}
  Decomposing $f,g$ according to the Efron-Stein decomposition over $(\Sigma^n,\mu^{\otimes n})$ as $f = \sum\limits_{S} f^{=S}$
  and $g = \sum\limits_{S} g^{=S}$, we get that
  \[
  \inner{f}{\mathrm{T}^{\otimes n} g}
  =\sum\limits_{S,Q}\inner{f^{=S}}{\mathrm{T}^{\otimes n} g^{=Q}}
  =\sum\limits_{S}\inner{f^{=S}}{\mathrm{T}^{\otimes n} g^{=S}}
  =\alpha\beta + \sum\limits_{S\neq \emptyset}\inner{f^{=S}}{\mathrm{T}^{\otimes n} g^{=S}}.
  \]
  The contribution from $\card{S}\leq d$ is at most
  \begin{align*}
  \sum\limits_{0<\card{S}\leq d}\card{\inner{f^{=S}}{\mathrm{T}^{\otimes n} g^{=S}}}
  \leq
  \sum\limits_{0<\card{S}\leq d}\norm{f^{=S}}_2\norm{\mathrm{T}^{\otimes n} g^{=S}}_2
  &\leq
  \sum\limits_{0<\card{S}\leq d}\norm{f^{=S}}_2\norm{g^{=S}}_2\\
  &\leq\sqrt{W_{\leq d}[f-\alpha]W_{\leq d}[g-\beta]}\\
  &\leq \frac{\alpha\beta}{10}.
  \end{align*}
  For $\card{S}>d$, we have that
  \[
  \norm{\mathrm{T}^{\otimes n} g^{=S}}_2
  \leq \lambda_2(\mathrm{T})^{\card{S}}\norm{g^{=S}}_2
  \leq (1-\Omega_q(1))^d\norm{g^{=S}}_2
  \leq \frac{\alpha\beta}{100}\norm{g^{=S}}_2
  \]
  for $d$ chosen suitably as in the statement. Thus, the contribution from $\card{S} > d$ is at most
  \[
  \sum\limits_{\card{S}>d}\card{\inner{f^{=S}}{\mathrm{T}^{\otimes n} g^{=S}}}
  \leq
  \sum\limits_{\card{S}>d}\norm{f^{=S}}_2\norm{\mathrm{T}^{\otimes n} g^{=S}}_2
  \leq
  \frac{\alpha\beta}{100}\sum\limits_{0<\card{S}\leq d}\norm{f^{=S}}_2\norm{g^{=S}}_2
  \leq \frac{\alpha\beta}{100}.
  \]
  Combining, we get that $\card{\sum\limits_{S\neq \emptyset}\inner{f^{=S}}{\mathrm{T}^{\otimes n} g^{=S}}}\leq \frac{\alpha\beta}{5}$, and so
  $\inner{f}{\mathrm{T}^{\otimes n} g}\geq \frac{4}{5}\alpha\beta$.
\end{proof}

\subsection{Random Restrictions}
We will make heavy use of random restrictions. Formally, given a function $f\colon (\Sigma^n,\mu^{\otimes n})\to\mathbb{C}$,
a restriction of it amounts to choosing $I\subseteq [n]$ a set of variables to remain alive and a setting $y\in\Sigma^{\overline{I}}$
for the rest of the variables. In that case, one gets the restricted function $f_{\overline{I}\rightarrow y}\colon\Sigma^{I}\to\mathbb{C}$
defined as
\[
f_{\overline{I}\rightarrow y}(z) = f(x_{I} = z, x_{\bar{I}} = y),
\]
where $(x_I = z, x_{\overline{I}} = y)$ is the point in $\Sigma^n$ whose $I$-coordinates are filled according to $z$, and whose $\bar{I}$-coordinates
are filled according to $y$.

A random restriction amounts to choosing either $I$, $y$ or both randomly. Fixing $I$, a random restriction could mean that the setting of $y$ is chosen
according to $\mu^{\overline{I}}$. Otherwise, we will have a probability parameter $q\in (0,1)$, in which case we choose $I\subseteq[n]$ randomly by including
each $i\in[n]$ in it with probability $q$ independently, and then choose $y\sim\mu^{\overline{I}}$. Whenever we apply random restrictions,
it will always be clear from context which variant we are using.

\skipi

The following simple lemma asserts that if a function $f$ is such that $f-\E[f]$ has significant weight on the low-degrees, then we can find a restriction
that leaves a considerable fraction of the variables alive and has significantly larger average than that of $f$. This will be very useful for us in the
density increment as it will tell us that once we find significant weight on the low-levels, we can automatically convert it to a density bump.
\begin{lemma}\label{lem:rest_to_bump}
  There exists $c>0$ such that the following holds.
  Suppose that $f\colon (\Sigma^n,\mu^{\otimes n})\to\{0,1\}$ is a function
  such that $W_{\leq d}[f-\E[f]]\geq \xi$ where $\xi>0$ and $d\in\mathbb{N}$ satisfy that $\xi\geq 2^{-c\frac{n}{d^2}}$.
  Then there is $I\subseteq [n]$ of size at least $\frac{1}{2d} n$ and $y\in\Sigma^{\overline{I}}$
  such that $\E[f_{\overline{I}\rightarrow y}]\geq \E[f] + \frac{\xi}{4e}$.
\end{lemma}
\begin{proof}
  Denote $\alpha = \E[f]$.
  Choose a random restriction $(I,y)$ by including each $i\in [n]$ in $I$ with probability $1/d$ and sampling $y\sim\mu^{\overline{I}}$,
  and define the random variable $Z_{I,y} = \E[f_{\overline{I}\rightarrow y}]$.
  Then
  \[
    \E_{I,y}[Z_{I,y}^2] = \Expect{I}{\Expect{y\sim\mu^{\overline{I}}}{\left(\sum\limits_{S} f^{=S}(y)1_{S\subseteq \overline{I}}\right)^2}}
    =\Expect{I}{\sum\limits_{S} \norm{f^{=S}}_2^21_{S\subseteq \overline{I}}}
    \geq \alpha^2 + \frac{1}{e}\sum\limits_{1\leq\card{S}\leq d} \norm{f^{=S}}_2^2,
  \]
  which is at least $\alpha^2 + \frac{1}{e}\xi$. Here, we used the fact that for each $S$ of size at most $d$, the probability that
  $S$ is contained in $\overline{I}$ is at least $1/e$. It follows from an averaging argument that with probability at least $\xi/2e$
  over the choice of $y,I$ we have $Z_{I,y}^2\geq \alpha^2 + \frac{1}{2e}\xi$.
  As $\card{I}\geq \frac{1}{2d} n$ with probability at least $1-2^{-\Omega(n/d^2)}$, we get
  that with probability at least $\frac{\xi}{2e} - 2^{-\Omega(n/d^2)}\geq \frac{\xi}{4e}$ both events hold,
  and we have
  \[
  Z_{I,y}\geq \sqrt{\alpha^2 + \frac{1}{2e}\xi}\geq \alpha\sqrt{1+\frac{\xi}{2e\alpha^2}}
  \geq \alpha\left(1+\frac{\xi}{4e\alpha}\right).
  \]
\end{proof}

The next lemma is very similar to the previous one, except that it applies to complex valued functions and the goal in it
is to show that a function with significant weight on the low-levels becomes somewhat biased after a random restriction.
It will be helpful for us whenever we try to get rid of low-degree functions $L$ so as to get correlations with product
functions, as explained in the introduction.
\begin{lemma}\label{lem:rest_to_correlation}
  There exists $c>0$ such that the following holds.
  Suppose that $g\colon (\Sigma^n,\mu^{\otimes n})\to\mathbb{C}$ is a $1$-bounded function
  such that $W_{\leq d}[g]\geq \xi$.
  Then choosing a random restriction $(I,y)$ where each $i\in [n]$ is included in $I$ with probability $1/2d$ and $y\sim \mu^{\overline{I}}$, we have that
  \[
        \Prob{I,y}{\card{\E[g_{\overline{I}\rightarrow y}]}\geq \sqrt{\xi/2e}}\geq \frac{\xi}{2e}.
  \]
\end{lemma}
\begin{proof}
  Define the random variable $Z_{I,y} = \E[f_{\overline{I}\rightarrow y}]$, and note that
  \[
    \E_{I,y}[\card{Z_{I,y}}^2] = \Expect{I}{\Expect{y\sim\mu^{\overline{I}}}{\card{\sum\limits_{S\neq \emptyset} g^{=S}(y)1_{S\subseteq \overline{I}}}^2}}
    =\Expect{I}{\sum\limits_{S} \norm{g^{=S}}_2^21_{S\subseteq \overline{I}}}
    \geq \frac{1}{e}\sum\limits_{0\leq \card{S}\leq d} \norm{f^{=S}}_2^2
    \geq \frac{\xi}{e}.
  \]
  As $\card{Z_{I,y}}\leq 1$ always, it follows that with probability at least $\frac{\xi}{2e}$ we have $\card{Z_{I,y}}\geq \sqrt{\frac{\xi}{2e}}$.
\end{proof}

\subsection{The CSP Stability Result}
Lastly, we need the following stability result from~\cite{BKMcsp4} discussed in the introduction; below is a formal statement.
\begin{thm}\label{thm:stab}
  For all $m\in\mathbb{N}$, $\alpha>0$ and $\eps>0$, there exists $d\in\mathbb{N}$ and $\delta>0$ such that the following holds.
  Suppose that $\mu$ is a distribution over $\Sigma\times\Gamma\times \Phi$ such that
  (1) the probability of each atom is at least $\alpha$,
  (2) the size of each one of $\Sigma,\Gamma,\Phi$ is at most $m$,
  (3) ${\sf supp}(\mu)$ is pairwise connected,
  and (4) $\mu$ does not admit non-trivial Abelian embeddings into $(\mathbb{Z},+)$.

  Then, if $f\colon\Sigma^n\to \mathbb{C}$, $g\colon\Gamma^n\to \mathbb{C}$, $h\colon\Phi^n\to \mathbb{C}$
  are $1$-bounded functions such that
  \[
  \card{\Expect{(x,y,z)\sim \mu^{\otimes n}}{f(x)g(y)h(z)}}\geq \eps,
  \]
  then there are $1$-bounded functions $u_1,\ldots,u_n\colon \Sigma\to \mathbb{C}$ and a function $L\colon \Sigma^n\to \mathbb{C}$ of degree at most $d$
  and $2$-norm at most $1$ such that
  \[
  \card{\Expect{x\sim \mu_x^{\otimes n}}{f(x)\cdot \overline{L(x)\prod\limits_{i=1}^{n}u_i(x_i)}}}\geq \delta.
  \]
  Furthermore, there are $r\in\mathbb{N}$ and Abelian group $(H,+)$ of size at most $r$ depending only on $m$,
  an Abelian embedding $(\sigma,\gamma,\phi)$ of $\mu$ into $H$ and characters $\chi_1,\ldots,\chi_n\in\hat{H}$
  such that $u_i(x_i) = \chi_i(\sigma(x_i))$.

  Quantitatively, one may take
  \[
        d = {\sf poly}_{m,\alpha}(1/\eps),
        \qquad
        \delta = 2^{-{\sf poly}_{m,\alpha}(1/\eps)}.
  \]
\end{thm}

\subsection{Classes of Functions}
In this section we discuss product functions $f(x_1,\ldots,x_n)$, which are functions that can be written as a product of functions
of absolute value $1$ of the individual coordinates.
\begin{definition}
  For an Abelian group $(H,+)$, a function $f\colon \Sigma^n\to\mathbb{C}$ is called a product function over $H$ if
  there are $1$-bounded functions $f_1,\ldots,f_n\colon \Sigma\to \{a\in\mathbb{C}~|~a^{\card{H}} = 1\}$ and
  a constant $c\in\mathbb{C}$ of absolute value $1$ such that
  \[
    f(x) = c\prod\limits_{i=1}^{n}f_i(x_i).
  \]
  We denote the set of product functions over $(H,+)$ with $n$-variables and alphabet $\Sigma$ by $\mathcal{P}(H,n,\Sigma)$.
\end{definition}
Often times, the alphabet $\Sigma$ will be clear from context, and we will drop it from the notation and denote the above by $\mathcal{P}(H,n)$;
we also denote by $\mathcal{P}(H)$ the union of all these collections. It is clear that the class $\mathcal{P}(H)$ is closed under
restrictions.

\subsection{Specialized Changes of Basis, Restricting the $z$-part and Closure Properties}
We will need to consider special type of changes of basis, defined below.
\begin{definition}
  Let $n'<n$ be integers. An $n'$-special basis for $\mathbb{F}_p^n$ is a basis $v_1,\ldots,v_{n'},u_1,\ldots,u_{n-n'}$
  in which $v_1,\ldots,v_{n'}$ have disjoint supports, and in their support the value of each coordinate is $1$.
\end{definition}

\begin{definition}\label{def:special_change_base}
  Let $n'<n$ be integers, $f\colon \mathbb{F}_p^n\to\mathbb{C}$ be a function and let $v_1,\ldots,v_{n'},u_1,\ldots,u_{n-n'}$
  be an $n'$-special basis for $\mathbb{F}_p^n$. We denote
  \[
  M_{\vec{u},\vec{v}}(x_1,\ldots,x_{n'},z_1,\ldots,z_{n-n'}) = \sum\limits_{i=1}^{n'} x_i v_i + \sum\limits_{i=1}^{n-n'} z_i u_i
  \]
  the corresponding change of basis transformation, and define the basis changed function $f^{\sharp}_{\vec{u},\vec{v}}\colon \mathbb{F}_p^n\to\mathbb{C}$
  as
  \[
    f^{\sharp}_{\vec{u},\vec{v}}(x_1,\ldots,x_{n'},z_1,\ldots,z_{n-n'})
    =f\left(M_{\vec{u},\vec{v}}(x_1,\ldots,x_{n'},z_1,\ldots,z_{n-n'})\right).
  \]
\end{definition}
We often remove the $v,u$ subscript whenever the specialized basis $v$ and $u$ is clear from context.
In conjunction to making a specialized change of basis, a key operation that we will often consider is restricting the $z$ part of it.
We refer to this operation as ``specialized change of basis and restricting the $z$-part'' henceforth.

A close inspection shows that
the class $\mathcal{P}(H)$ is also closed under specialized changes of basis and restricting the $z$-part, and we will use this fact numerous times.
For the sake of completeness, we include a proof of this fact below.
\begin{claim}
  Let $(H,+)$ be an Abelian group. Then the class $\mathcal{P}(H)$ is closed under specialized changes of basis and restricting the $z$-part.
\end{claim}
\begin{proof}
  Let $P\in \mathcal{P}(\mathbb{F}_p, H,n)$, let $v_1,\ldots,v_{n'},u_1,\ldots,u_{n-n'}$ be a specialized basis
  and let $M_{\vec{v},\vec{u}}$ be the corresponding change of basis transformation. Write $P(x) = \prod\limits_{i=1}^{n} f_i(x_i)$
  and denote $w = M_{\vec{v},\vec{u}}(x,z)$.
  For $j=1,\ldots,n'$ and $i\in {\sf supp}(v_j)$ we have that $w_i = x_j + L_i(z)$
  where $L_i(z)$ is a linear function of $z$. For $i\not\in\bigcup_{j}{\sf supp}(v_j)$, we have that $w_i = L_i(z)$
  where again $L_i(z)$ is a linear function of $z$.
  It follows that
  \[
  (P\circ M_{\vec{v},\vec{u}})(x,z)
  =P(w)
  =\prod\limits_{i=1}^{n'} f_i(x_{j_i} + L_i(z))
  \cdot C(z),
  \]
  where $j_i$ is the unique $j$ such that $i\in {\sf supp}(v_j)$, and $C(z)$ is a complex number depending only on
  $z$. It follows that
  \[
  (P\circ M_{\vec{v},\vec{u}})|_{[n]\setminus[n']\rightarrow z}(x)
  =C(z)\prod\limits_{j=1}^{n'}f_j(x_j)
  \]
  where $f_j(x_j) = \prod\limits_{i\in {\sf supp}(v_j)} f_i(x_{j_i}+L_i(z))$, so
  $(P\circ M_{\vec{v},\vec{u}})|_{[n]\setminus[n']\rightarrow z}\in \mathcal{P}(\mathbb{F}_p, H, n')$.
\end{proof}

Another important property of specialized changes of basis that we need is that, after restricting the $z$-part, they preserve functions that are restricted $3$-AP free.
\begin{claim}\label{claim:special_pres_free}
  Let $p$ be a prime, let $1<n'<n$, and let $f\colon\mathbb{F}_p^n\to \{0,1\}$ be a restricted $3$-AP free function.
  Then for every $n'$-specialized basis $v_1,\ldots,v_{n'}$, $u_1,\ldots,u_{n-n'}$, looking at the
  function $f_{\vec{u},\vec{v}}^{\sharp}$ from Definition~\ref{def:special_change_base}, it holds that for all $z\in\mathbb{F}_p^{n-n'}$,
  $(f_{\vec{u},\vec{v}}^{\sharp})_{[n]\setminus [n']\rightarrow z}$ is a restricted $3$-AP free function.
\end{claim}
\begin{proof}
  Suppose that this is not the case. Then there are $x,z$ and $a\in\{0,1\}^{n'}$ not identically $0$
  such that $f_{\vec{u},\vec{v}}^{\sharp}(x,z) = f_{\vec{u},\vec{v}}^{\sharp}(x+a,z) = f_{\vec{u},\vec{v}}^{\sharp}(x+2a,z) = 1$,
  which gives the following restricted $3$-AP in $f$:
  $x', x' + a', x' + 2a'$ where $x' = \sum\limits_{i=1}^{n'}x_i v_i + \sum\limits_{i=1}^{n-n'}z_iu_i$
  and $a' = \sum\limits_{i=1}^{n'} a_i v_i\in \{0,1\}^n\setminus\set{0}$.
\end{proof}

\section{Density Increment Tools}
In this section we present some tools that will help us in the density increment argument.
Specifically, in our case we will have a function $f\colon \mathbb{F}_p^n\to\{0,1\}$ such that for $\tilde{f} = f-\E[f]$ it holds that
\begin{equation}\label{eq:dit_1}
\card{\E_{x,a}[\tilde{f}(x)f(x+a)f(x+2a)]}\geq \eps.
\end{equation}
From this fact, we wish to conclude from that after some operations (which for us will be restrictions and changes of basis), we can find
a function $g\colon \mathbb{F}_p^{n'}\to\{0,1\}$ such that:
\begin{enumerate}
  \item $\E[g]\geq \E[f] + \Omega_{\eps}(1)$
  \item If $f$ is restricted $3$-AP free, then $g$ is restricted $3$-AP free.
  \item $n'$ is not too much smaller than $n$.
\end{enumerate}
Towards this end, Theorem~\ref{thm:stab} gives some explanation to why this is the case, asserting that $\tilde{f}$ is correlated
with a function of the form $P(x)\cdot L(x)$ where $P\in\mathcal{P}(H)$ and $L$ is a low-degree function. In this section, we
develop tools that convert such information into a density increment from $f$ as described above.

\paragraph{High level idea, ignoring the low-degree part.}
Ignoring the low-degree part, the basic idea is to partition the space $\mathbb{F}_p^n$ into parts
on which the function $P$ is constant. Given such partition, the fact that $f$ is correlated with
$P$ says that the average of $f$ cannot be the same on all parts, and therefore there must be some variance in it.
In particular, there must be significant number of parts where the density of $f$ noticeably increases, hence we get
a density increment for $f$ satisfying the first and third item above.
To get the second property, however, we need to chosen a partition that interacts well with the problem in
hand. In the standard $3$-AP free set problem, it suffices that the partition would be into subspaces.
However, in our case of restricted $3$-AP free sets, more careful structure is needed in order to
preserve the property of restricted differences.

This is the point in the argument where specialized changes of basis enter. If $v_1,\ldots,v_{n'}$,
$u_1,\ldots,u_{n-n'}$ is a specialized basis as in Definition~\ref{def:special_change_base}, then any
tuple of points of the form $\sum\limits_{i}x_iv_i + \sum\limits_{j}z_ju_j$,
$\sum\limits_{i=1}^{n'}(x_i+a_i)v_i + \sum\limits_{j=1}^{n-n'}z_ju_j$
and
$\sum\limits_{i=1}^{n'}(x_i+2a_i)v_i + \sum\limits_{j=1}^{n-n'}z_ju_j$
where $a_i\in\{0,1,2\}$ for all $i$ is a restricted $3$-AP.
Thus, we can consider the partition of $\mathbb{F}_p^n$ induced by
the affine subspaces $\{W_z\}_{z\in\mathbb{F}_p^{n-n'}}$ where $W_z$ is defined as
\[
W_z = \sett{\sum\limits_{i=1}^{n'}x_iv_i + \sum\limits_{j=1}^{n-n'}z_ju_j}{x_1,\ldots,x_{n'}\in\mathbb{F}_p}.
\]
The restriction of $f$ to each $W_z$ can be viewed as a function on $\mathbb{F}_p^{n'}$ which is restricted $3$-AP free (if $f$ is restricted $3$-AP)
free. Thus, specialized changes of basis will give us the type of partitions we need to facilitate a density increment argument. To carry out our argument,
though, we will need to robustify the correlation between $f$ and $P$. By that, we mean that the restrictions of $f$ and $P$ to $W_z$ where $z$ is chosen
randomly remain correlated with probability close to $1$, and in Lemma~\ref{lem:list_decode} we show that this is achievable.

\paragraph{Re-introducing the low-degree part.}
In our setting though, we have to also address the low-degree part and to handle them we first apply random restrictions. Roughly speaking,
we show that if $f$ is correlated with $P\cdot L$, then after a suitable random restriction, with noticeable probability $f$ is correlated with $P'$,
where $P'$ is a restriction of $P$, and then use the above idea. A subtle point one always has to keep in mind is that we never wish the density
of $f$ to drop by too much as a result of a random restriction -- indeed we want it to roughly be the same, so that using the above idea we will
eventually get a density increment. This is easy to handle, as we show that if the density of $f$ noticeably drops with significant probability
as a result of a random restriction, then it must be the case that its density also noticeably increases with significant probability, in which
case we are automatically done finding a density increment.

\subsection{Random Restrictions and Correlations}
The following lemma asserts that if $f$ is correlated with a function of the form $P\cdot L$ where $\norm{L}_2\leq 1$ and $L$ is a low-degree function,
then after random restriction $f$ is correlated with $P'$ where $P'$ is a restriction of $P$.
\begin{lemma}\label{lem:find_corr_restrict}
  Suppose that $f\colon (\Sigma^n,\mu^{\otimes n})\to\mathbb{C}$ is $1$-bounded,
  $P\in\mathcal{P}(H,n,\Sigma)$ and
  $L\colon \Sigma^n\to\mathbb{C}$ has degree at most $d$ and $2$-norm at most $1$.
  If $\card{\inner{f}{P\cdot L}}\geq \delta$, then taking $(I,y)$ a random restriction that fixes
  a coordinate with probability $1-1/2d$, with probability at least $\frac{\delta^2}{2e}$,  we
  have that
  \[
  \card{\inner{f_{I\rightarrow y}}{P|_{I\rightarrow y}}}\geq  \frac{\delta}{\sqrt{2e}}.
  \]
\end{lemma}
\begin{proof}
  We have
  \[
  \delta
  \leq
  \card{\inner{f \overline{P}}{L}}
  =
  \card{\inner{(f \overline{P})^{\leq d}}{L}}
  \leq \norm{(f \overline{P})^{\leq d}}_2\norm{L}_2
  \leq \norm{(f \overline{P})^{\leq d}}_2
  \]
  where we used Cauchy-Schwarz and the fact that the $2$-norm of $L$ is at most $1$. Thus,
  $W_{\leq d}[f \overline{P}]\geq \delta^2$. Using Lemma~\ref{lem:rest_to_correlation} now
  we get that choosing a random restriction as in the statement of the lemma,
  with probability at least $\frac{\delta^2}{2e}$ we have that
  \[
  \frac{\delta}{\sqrt{2e}}
  \leq
  \card{\E[f_{\overline{I}\rightarrow y} \overline{P}_{\overline{I}\rightarrow y}]}
  =
  \card{\inner{f_{\overline{I}\rightarrow y}}{P_{\overline{I}\rightarrow y}}}.
  \qedhere
  \]
\end{proof}

\subsection{Robustifying Correlations}\label{sec:list_decode}
To facilitate a density increment argument, we need to robustify the correlation of $f$ with $P$, and in fact
the robustification we need is with respect to a somewhat more complicated $2$-step random restriction process.
In the most simplistic form of robustifying correlations (which is insufficient for our purposes),
one wants the correlation $\card{\inner{f}{P}}$ to almost always remain bounded away from $0$ under applications random restrictions.
Indeed, such robustification is possible to achieve by passing to restrictions of $f$ and $P$ so long as one can increase the correlation
by doing so. The idea is that if we have functions $f$, $P$
for which the correlation of random restriction drops to be close to $0$ with noticeable probability, then with noticeable probability the
correlation of random restrictions
must also increase by a significant amount.

The more complicated form of robustification that we need proceeds as follows: one first applies any specialized change of basis,
then restricts the $z$-part according to some setting $z$ randomly. After that, a large set of coordinates $I_z$ is chosen out of the live
variables by an adversary, and then the variables outside $I_z$ are fixed randomly according to $z'$. The notion of robust correlation
we need asserts that, even after such restriction process, except for small probability, the correlation between the restrictions
of $f$ and $P$ is still noticeable. Below is a formal statement.
\begin{lemma}\label{lem:list_decode}
  For all $p\geq 3$ prime, $r\in\mathbb{N}$, $\eps>0$ and $\delta>0$, there are $\beta_0$ and $\gamma>0$ such that the following holds
  for all $0<\beta\leq \beta_0$.
  Let $(H,+)$ be a finite Abelian group of size at most $r$,
  let $f\colon (\mathbb{F}_p^n,\mu^{\otimes n})\to [-1,1]$ be a function
  where $\mu$ is the uniform distribution over $\mathbb{F}_p$,
  let $P\in\mathcal{P}(H,n,\mathbb{F}_p)$
  and suppose that
  $\card{\inner{f}{P}}\geq \eps$.

  Then at least one of the following cases holds:
  \begin{enumerate}
    \item Density bump: there is $n'\geq \gamma n$ and $f'\colon \mathbb{F}_p^{n'}\to[-1,1]$ which is the result of a sequence of restrictions,
    and specialized changes of basis + restricting the $z$-part on $f$, such that $\E[f']\geq \E[f] + \beta$.
    \item There is $n'\geq \gamma n$,
    a function $f'\colon \mathbb{F}_p^{n'} \to [-1,1]$ and
    $P'\in \mathcal{P}(H,n',\mathbb{F}_p)$, which are the result of a sequence of restrictions
    and specialized changes of basis + restricting the $z$-part on $f$ and $P$, such that
    \begin{enumerate}
      \item For all $n''\geq n'/r^{10p}$, for all $n''$-special basis and $v_1,\ldots,v_{n''}, u_1,\ldots,u_{n'-n''}$,
      letting $M_{\vec{u},\vec{v}}$ be the basis change transformation and looking at
      $f'^{\sharp} = f'\circ M_{\vec{u},\vec{v}}$
      and
      $P'^{\sharp} = P'\circ M_{\vec{u},\vec{v}}$,
      \[
      \Prob{\substack{z\sim \mu^{n'-n''}\\ z'\sim \mu^{n''}}}
      {
      \exists I_z\subseteq [n']\setminus [n'-n''], \card{I_z}\geq p^{-100r}n'',
      \card{
      \inner{f'^{\sharp}_{\substack{[n'-n'']\rightarrow z\\\bar{I_z}\rightarrow z'_{\bar{I_z}}}}}
      {P'^{\sharp}_{\substack{[n'-n'']\rightarrow z\\\bar{I_z}\rightarrow z'_{\bar{I_z}}}}}
      }
      \leq \frac{\eps}{2}
      }
      \leq \delta.
      \]
      \item The average of $f'$ is almost as large as of $f$: $\E[f']\geq \E[f] - \sqrt{\beta}$.
    \end{enumerate}
  \end{enumerate}
  Quantitatively, one can take
  \[
  \gamma = \left(p^{-100r} r^{-10p}\right)^{\frac{4}{\delta\eps}},
  \qquad
  \beta_0 = \left(\frac{\eps \delta}{8}\right)^{\frac{40}{\delta\eps}}.
  \]
\end{lemma}
\begin{proof}
  First, we pick parameters
  \[
  N = \frac{4}{\delta \eps},
  \qquad
  \beta_0 = \left(\frac{\eps \delta}{8}\right)^{10N},
  \qquad
  \gamma = \left(p^{-100r} r^{-10p}\right)^{N}
  \]
  If the first bullet holds for $\gamma$ then we are done, so assume otherwise. Henceforth $0<\beta\leq \beta_0$.
  We prove that then the second bullet holds. The argument will be iterative, and we will construct a sequence of functions
  $f_1,\ldots,f_m$ where $f_{j+1}$ is a result of restrictions and specialized changes of basis + restricting the $z$ part applied on $f_j$,
  as well as
  functions $P_1,\ldots,P_m$ where $P_{j+1}$ is a result of restrictions and specialized changes of basis + restricting the $z$ part applied on $P_{j}$.
  Our iterative process starts with $f_1 = f$ and $P_1=P$.
  We show that if $f_j,P_j$ violate the second bullet, then we may construct from them
  $f_{j+1}$ and $P_{j+1}$ such that $\E[f_{j+1}] \geq \E[f]-\beta^{1-2j/N}$
  and
  $\card{\inner{f_{j+1}}{P_{j+1}}}\geq \card{\inner{f_{j}}{P_{j}}}+j\frac{\eps\delta}{4}$.
  As $\card{\inner{f_{j+1}}{P_{j+1}}}\leq 1$ always, the process must terminate
  after at most $N$ steps.

  Suppose we have $f_j, P_j$ that violate the second bullet, let $n_j$ be the number of variables,
  and take $n''\geq r^{-10pm} n_j$ for which the condition is violated. Thus, we may find an $n''$-specialized change of basis
  $v_1,\ldots,v_{n''}$ and $u_1,\ldots,u_{n_j-n''}$ for which the second item fails, and we denote
  $f_j^{\sharp} = f_j\circ M_{\vec{u},\vec{v}}$ and $P_j^{\sharp} = P_j\circ M_{\vec{u},\vec{v}}$, as well as $n''' = p^{-100r} n''$.
  We say $z\in \mathbb{F}_p^{n'-n''}$ is bad if there exists $I_z$ such that in the event in the second item holds.
  For each bad $z$ we fix $I_z$ as therein. If $z$ is not bad, we set $I_z = \emptyset$.

  Note that $\E[f_j^{\sharp}] = \E[f_j] \geq \E[f] - \beta^{1-2j/N}$. Further note that
  \[
  \Expect{z\sim \mu^{n_j-n''}, z'\sim \mu^{n''}}{\E[f'^{\sharp}_{\substack{[n_j-n'']\rightarrow z\\\bar{I_z}\rightarrow z'_{\bar{I_z}}}}]}
  =
  \Expect{z\sim \mu^{n_j-n''}}{\E[f'^{\sharp}_{[n_j-n'']\rightarrow z}]}
  =
  \E[f_j^{\sharp}]
  \geq \E[f] - \beta^{1-2j/N}.
  \]
  Let $E_1$ be the event that
  $\E[{f_j}^{\sharp}_{\substack{[n_j-n'']\rightarrow z\\\bar{I_z}\rightarrow z'_{\bar{I_z}}}}]\leq \E[f] - \beta^{1-2(j+1)/N}$.
  If
  $\E[{f_j}^{\sharp}_{\substack{[n_j-n'']\rightarrow z\\\bar{I_z}\rightarrow z'_{\bar{I_z}}}}]\geq \E[f] + \beta$
  for some $z$ and $z'$, then we are
  done as we found a function as in the first item. Thus, we assume henceforth that there are no such $z,z'$.

  By Markov's inequality
  \[
  \Prob{}{E_1}
  =
  \Prob{z,z'}{
  \E[f]+\beta - \E[{f_j}^{\sharp}_{\substack{[n'-n'']\rightarrow z\\\bar{I_z}\rightarrow z'_{\bar{I_z}}}}]
  \geq \beta + \beta^{1-2(j+1)/N}}
  \leq \frac{\beta+\beta^{1-2j/N}}{\beta + \beta^{1-2(j+1)/N}}
  \leq 2\beta^{2/N}.
  \]
  Next, consider the random variable $Y_{z,z'} = \inner{{f_j}^{\sharp}_{\substack{[n'-n'']\rightarrow z\\\bar{I_z}\rightarrow z'_{\bar{I_z}}}}}{{P_j}^{\sharp}_{\substack{[n'-n'']\rightarrow z\\\bar{I_z}\rightarrow z'_{\bar{I_z}}}}}$, and let $E_2$ be the event that $\card{Y_{z,z'}}\geq \card{\inner{f_j}{P_j}} + \frac{\eps\delta}{4}$.
  We note that
  \[
  \Expect{z,z'}{\card{Y_{z,z'}}}
  \geq
  \card{\Expect{z,z'}{Y_{z,z'}}}
  =\card{\inner{f_j}{P_j}}.
  \]
  As $\card{\inner{f_j}{P_j}} \geq \eps$ and by assumption $\card{Y_{z,z'}}\leq \eps/2$ with probability at least $\delta$, it follows that
  $\card{Y_{z,z'}}\geq \card{\inner{f_j}{g_j}} + \frac{\eps\delta}{4}$ with probability at least $\frac{\eps\delta}{4}$. Thus, $\Prob{}{E_2}\geq \frac{\eps\delta}{4}$.

  We get that $\Prob{}{\overline{E_1}\cap E_2}\geq \frac{\eps\delta}{4}-2\beta^{2/N} > 0$, and so we may find $z,z'$ for which both $\overline{E_1}$ and
  $E_2$ hold, and these $z,z'$ give us $f_{j+1} = {f_j}^{\sharp}_{\substack{[n_j-n'']\rightarrow z\\\bar{I_z}\rightarrow z'_{\bar{I_z}}}}$
  and $P_{j+1} = {P_j}^{\sharp}_{\substack{[n_j-n'']\rightarrow z\\\bar{I_z}\rightarrow z'_{\bar{I_z}}}}$ as required.
\end{proof}

\section{Proof of Main Results}
With the density increment tools in hand, we can now establish the basic density increment lemma.
Throughout this section, $\mu$ is the distribution of $(x,x+a,x+2a)$ where $x\sim\mathbb{F}_p$ and $a\sim\{0,1,2\}$
are sampled uniformly and independently. We begin with a few basic properties of $\mu$.

The first claim asserts that $\mu$ is pairwise connected, and here we use the fact that $p$ is a prime.
\begin{claim}\label{claim:mu_pairwise}
  If $p$ is prime, then $\mu$ is pairwise connected.
\end{claim}
\begin{proof}
  For $a\in \{1,2\}$, consider the bipartite graphs $G_{a} = (\mathbb{F}_p\times\{L,R\}, E_a)$
  where
  \[
  E_a = \sett{((x,L),(y,R))}{x = y \text{ or } y = x+a}.
  \]
  Note that the fact that $\mu$
  is pairwise connected is equivalent to the fact that the graphs $G_a$ are connected.
  The fact that $G_a$ is connected follows since each $a\in \{1,2\}$ generates $\mathbb{F}_p$
  by the operation of addition (as they are both invertible under multiplication in $\mathbb{F}_p$).
\end{proof}

Next, we show that $\mu$ does not admit any non-trivial Abelian embeddings into $(\mathbb{Z},+)$.
This is the part of the proof where we need the common difference to be from $\{0,1,2\}^n$
(as opposed to be from $\{0,1\}^n$).
\begin{claim}\label{claim:mu_not_Z}
  The distribution $\mu$ admit no non-trivial Abelian embeddings into $(\mathbb{Z},+)$.
\end{claim}
\begin{proof}
  Suppose that
  $\sigma\colon \mathbb{F}_p\to\mathbb{Z}$,
  $\gamma\colon \mathbb{F}_p\to\mathbb{Z}$ and
  $\phi\colon \mathbb{F}_p\to\mathbb{Z}$
  form an Abelian embedding of $\mu$ into $(\mathbb{Z},+)$.
  Then for all $x\in\mathbb{F}_p$ we have that
  $(x,x+1,x+2),(x,x,x)\in {\sf supp}(\mu)$ and
  so
  $\sigma(x) + \gamma(x+1) + \phi(x+2) = 0 = \sigma(x) + \gamma(x) + \phi(x)$.
  Hence, denoting $\partial_a \sigma(x) = \sigma(x+a) - \sigma(x)$ and similarly for $\phi$,
  we get that $\partial_1\gamma(x) + \partial_2 \phi(x)=0$.

  Also, we have that $(x-2,x-1,x), (x-2,x,x+2)\in {\sf supp}(\mu)$ for all $x\in\mathbb{F}_p$,
  hence
  $\sigma(x-2) + \gamma(x-1) + \phi(x) = 0 = \sigma(x-2) + \gamma(x) + \phi(x+2)$.
  rearranging gets $\partial_1\gamma(x-1) + \partial_2\phi(x)=0$.

  Combining the two equations, we get that $\partial_1\gamma(x) = \partial_1\gamma(x-1)$, hence
  $\partial_1\gamma(x)$ is a constant function. Noting that $\sum_{x\in\mathbb{F}_p} \partial_1\gamma(x) = 0$,
  we get that $\partial_1\gamma(x) \equiv 0$, and so $\gamma$ is constant.

  We conclude that $\partial_2 \phi(x) = 0$, so $\phi(x+2) = \phi(x)$ for all $x$. As $p$ is prime it follows that
  $\phi$ is also a constant function, and thus $\sigma$ is also constant function.

  We conclude that each one of $\sigma,\gamma$ and $\phi$ is constant, hence
  $\mu$ admit no non-trivial Abelian embeddings into $(\mathbb{Z},+)$.
\end{proof}

We also need the following auxiliary straightforward fact.
\begin{fact}\label{fact:trivial}
  Let $(H,+)$ be an Abelian group of size $r$, and let $h\colon \Sigma\to\{a~|~a^r = 1\}$.
  Then $h^{r}\equiv 1$.
\end{fact}

\subsection{The Basic Density Increment Argument}
In the next lemma we combine Theorem~\ref{thm:stab} and Lemma~\ref{lem:list_decode}
to conclude a basic density increment argument, asserting that if $f\colon \mathbb{F}_p^{n}\to\{0,1\}$
is restricted $3$-AP free function of density $\alpha$, then we may find a restricted $3$-AP function
$f\colon \mathbb{F}_p^{n'}\to\{0,1\}$ with significantly larger density and $n'\geq \gamma(\alpha) n$.
\begin{lemma}\label{lem:density_inc_step}
  For all $\alpha>0$ there is $\alpha'>0$ and $\gamma>0$ such that the following holds.
  Suppose $f\colon \mathbb{F}_p^n\to\{0,1\}$ is a function with average $\alpha$ that is
  restricted $3$-AP free. Then there exists $n'\geq \gamma n$ and $g\colon\mathbb{F}_p^{n'} \to \{0,1\}$
  which is restricted $3$-AP free and $\E[g]\geq \alpha + \alpha'$.

  Quantitatively, we have $\alpha' = \gamma \geq {\sf exp}(-C{\sf exp}(1/\alpha^C))$
  where $C = C(p)>0$ is a constant depending only on $p$.
\end{lemma}
\begin{proof}
Let $f\colon \mathbb{F}_p^n\to\{0,1\}$ be the indicator function of $A$. As $A$ is restricted $3$-AP free it follows that
  the only triplets $x$, $x+a$ and $x+2a$ in $A$ where $x\in \mathbb{F}_p^n$ and $a\in\{0,1,2\}^n$ are such that $a=\vec{0}$
  and $x\in A$, hence
  \[
  \Expect{x\in\mathbb{F}_p^n,a\in\{0,1,2\}^n}{f(x)f(x+a)f(x+2a)} = 3^{-n} \alpha.
  \]
  Consider the expectation $W = \Expect{x\in\mathbb{F}_p^n,a\in\{0,1,2\}^n}{f(x+a)f(x+2a)}$.
  If $W\leq \alpha^2/100$, then by Lemma~\ref{lem:small_wt_to_neg} for $d = O_p(\log(1/\alpha))$
  we have that $W_{\leq d}[f-\alpha]\geq \Omega(\alpha^4)$. By Lemma~\ref{lem:rest_to_bump} it follows
  that there is $I$ of size at least $n/2d$ and a restriction $y\in\mathbb{F}_p^{\overline{I}}$ such
  that $\E[f_{\overline{I}\rightarrow y}]\geq \alpha + \Omega(\alpha^4)$, and we found a restriction as
  required in the statement of the lemma.

  We henceforth assume that $W\geq \alpha^2/100$. Thus,
  \begin{equation}\label{eq:density_increment1}
  \card{\Expect{x\in\mathbb{F}_p^n,a\in\{0,1,2\}^n}{(f(x)-\alpha)f(x+a)f(x+2a)}}
  =\card{3^{-n}\alpha - \alpha W}\geq \alpha^3/200.
  \end{equation}
  By Claim~\ref{claim:mu_pairwise} the distribution $\mu$ is pairwise connected and by Claim~\ref{claim:mu_not_Z} it doesn't admit any non-trivial $(\mathbb{Z},+)$ embeddings,
  so we may apply Theorem~\ref{thm:stab}.
  Let $r$ and $(H,+)$ be from that theorem (that only depends on $p$); without loss of generality we assume that $\card{H} = r$, otherwise we decrease $r$.
  We take $D$ and $\delta>0$ from Theorem~\ref{thm:stab} for $\eps = \frac{\alpha^3}{200}$. Thus, by~\eqref{eq:density_increment1} it follows
  that there is $P\in \mathcal{P}(H, n, \mathbb{F}_p)$ and $L\colon \Sigma^n\to\mathbb{C}$ of degree at most $D$ and $2$-norm at most $1$ such that
  $\card{\inner{f-\alpha}{P\cdot L}}\geq \delta$. Denote $\tilde{f} = f-\alpha$.

  Take $\beta_0$ and $\gamma$ from Theorem~\ref{lem:list_decode} for the parameters $r$ and the parameters ``$\eps$'' there being
  $\eps'=\delta/\sqrt{2e}$ and ``$\delta$''
  there being both equal to $p^{-10r}/10r$.
  Let $\eta = \min(\frac{\beta_0^2}{1000},\frac{\delta^{100}}{100})$, and let $(I,y)$ be a $1/2d$ random restriction according to $\mu$. Let $E$ be the event that
  $\E[f_{\overline{I}\rightarrow y}]\leq \alpha - \eta$. If $\Prob{}{E}\geq \eta$, then
  \[
  \Prob{}{\E[f_{\overline{I}\rightarrow y}]\geq \alpha + \eta^2/2}\geq \frac{\eta^2}{2},
  \]
  and we may find a restriction as needed in the lemma. Henceforth, we assume that $\Prob{}{E}\leq \eta$. By Lemma~\ref{lem:find_corr_restrict}
  we get that $\card{\inner{\tilde{f}_{\overline{I}\rightarrow y}}{P_{\overline{I}\rightarrow y}}}\geq \frac{\delta}{\sqrt{2e}}$ with probability at least $\frac{\delta^2}{2e}$,
  and it follows that with probability at least $\frac{\delta^2}{2e}-\eta \geq \frac{\delta^2}{4e}$ we have that $(I, y)$ satisfies that
  \[
  \card{\inner{\tilde{f}_{\overline{I}\rightarrow y}}{P_{\overline{I}\rightarrow y}}}\geq \eps',
  \qquad
  \E[f_{\overline{I}\rightarrow y}]\geq \alpha - \eta.
  \]
  Thus, we may find $I$ of size at least $n/2D$ and $y$ satisfying both of these conditions.
  We denote $f' = f_{\overline{I}\rightarrow y}$, $P' = P_{\overline{I}\rightarrow y}$, $n' = \card{I}$ and $\tilde{f}' = \tilde{f}_{\overline{I}\rightarrow y}$.
  Thus, we get that
  \begin{equation}\label{eq:density_increment2}
    \card{\inner{f'}{P'}}\geq \eps',
    \qquad
    \E[\tilde{f}']\geq -\eta.
  \end{equation}

  We now apply Lemma~\ref{lem:list_decode} on $\tilde{f}'$ and $P'$ and $\beta = \min(\beta_0^2/100, 2\eta)$, and there are two cases depending on which of the items therein holds.
  \paragraph{The case the first item holds.}
  If the first item holds, then we may find $f^{\sharp}$ which is a result of applying specialized changes of basis and
  restrictions of the $z$ part on $\tilde{f}'$ such that $\E[f^{\sharp}]\geq \E[\tilde{f}'] + \beta \geq \beta - \eta\geq \beta/2$,
  and so $g = \alpha + \tilde{f}'$ is a result of applying specialized changes of basis and
  restrictions of the $z$ part on $f$, and $\E[g]\geq \alpha + \beta/2$. It follows by Claim~\ref{claim:special_pres_free}
  that $g$ is a function as required in the statement of the lemma.

  \paragraph{The case the second item holds.} If the second item holds, then we may find $f^{\sharp}$ and $P^{\sharp}$ which are a result of applying specialized changes of basis and
  restrictions of the $z$ part on $\tilde{f}'$ and $P'$, so that $\E[f^{\sharp}]\geq \E[\tilde{f}'] - \sqrt{\beta}\geq -3\sqrt{\eta}$ and the second
  item therein holds. Denote the number of variables depend on by $n^{\sharp}$, so that $n^{\sharp}\geq \gamma n$. As $\mathcal{P}(H,n,\mathbb{F}_p)$
  is closed under specialized changes it follows that $P^{\sharp}\in \mathcal{P}(H, n^{\sharp}, \mathbb{F}_p)$, and so we may write
  (note that we may assume the constant to be $1$, as otherwise we may multiply by its complex conjugate and get that the absolute value
  of the correlation with $f$ remains the same)
  \[
  P^{\sharp}(x) = \prod\limits_{i=1}^{n^{\sharp}} h_i(x_i)
  \]
  where $h_i\colon \Sigma\to \{a~|~a^{r} = 1\}$. The total number of distinct tuples
  $h_i$'s is $p^{r}$, hence
  by the pigeonhole principle we may find $h$
  such that $h_i = h$ for at least $n' = \frac{n^{\sharp}}{p^r}$
  of the coordinates, and we let this set of coordinates be denoted by $R$.
  Without loss of generality we assume that $\card{R}$ is divisible by $r$, as otherwise we may drop from it at most $r-1$
  elements to make it divisible by $r$. Also, without loss of generality assume that $R = [n']$.

  Next, we set up a specialized change of basis. Towards this end, we partition $R$ into $R_1,\ldots, R_{n'/r}$ where each set is of
  size precisely $r$, and choose for each $i=1,\ldots, n'/r$ a vector $v_i = 1_{R_i}\in \{0,1\}^{n^{\sharp}}$, and then complete
  $\{v_1,\ldots,v_{n'/r}\}$ into a basis of $\mathbb{F}_p^{n^{\sharp}}$ by adding vectors $u_1,\ldots,u_{n^{\sharp} - n'/r}$.
  Denote $J = [n^{\sharp}]\setminus [n'/r]$, and write $w = M_{\vec{u},\vec{v}}(x,z)$ where $M_{\vec{u},\vec{v}}$ is the specialized
  change of basis matrix corresponding to the basis $\vec{u},\vec{v}$ we constructed. We are going to randomly restrict the $z$
  part according to $z\sim \mathbb{F}_p^{n^{\sharp}-n'/r}$, and towards that end we first note that
  for $i=1,\ldots, n'$, we have that $w_i = x_{j} + H_i(z)$ where $j$ is the unique index such that $i\in R_j$, and
  $H_i(z)$ is a linear function in the $z$'s. We say $j\in [n'/r]$ is good if, for our choice of $z$, it holds that $H_i(z) = 0$ for all $i\in R_j$.
  Note that for each $j$, the probability that $j$ is good under the choice of $z$ is at least $p^{-r}$, and so the expectation of the number
  of good $j$'s is at least $p^{-r} \frac{n'}{r}$. Thus, by an averaging argument the probability that the number of good $j$'s is at least $\frac{p^{-r}}{2}\frac{n'}{r}$
  is at least $p^{-r}/2$, and we denote this event by $E$. For each $z$ we choose the set $I_z = \cup_{j\text{ good}} R_j$.

  We now choose $z\sim \mathbb{F}_p^{n^{\sharp}-n'}$ and $z'\sim\mathbb{F}_p^{n'}$.
  Using the second item of Theorem~\ref{lem:list_decode} now with the choice of $I_z$, we get that
  \begin{equation}\label{eq:second_item_holds}
  \card{
  \inner
  {f^{\sharp}\circ M_{\vec{u},\vec{v}}|_{\substack{J\rightarrow z\\ \overline{I_z}\rightarrow z'_{\overline{I_z}}}}}
  {P^{\sharp}\circ M_{\vec{u},\vec{v}}|_{\substack{J\rightarrow z\\ \overline{I_z}\rightarrow z'_{\overline{I_z}}}}}
  }\geq \frac{\eps'}{2}
  \end{equation}
  with probability at least $1-\frac{p^{-10r}}{10r}$. Thus, with probability at
  least $1-\frac{p^{-5r}}{\sqrt{10r}}$ over $z$, we have that
  \begin{equation}\label{eq:second_item_holds2}
  \Prob{z'}{\eqref{eq:second_item_holds}\text{ holds}}\geq 1-\frac{p^{-5r}}{\sqrt{10r}}.
  \end{equation}
  As the probability of $E$ is at least $p^{-r}/2$, we
  get that with probability at least $p^{-r}/2-\frac{p^{-5r}}{\sqrt{10r}}\geq p^{-r}/4$ both $E$ and~\eqref{eq:second_item_holds2}
  hold. We call such $z$ good.

  If
  $\E[f^{\sharp}\circ M_{\vec{u},\vec{v}}|_{J\rightarrow z}]\geq \eta$, then we can take $g = \alpha + f^{\sharp}\circ M_{\vec{u},\vec{v}}|_{J\rightarrow z}$
  and get that $g$ is $0,1$ valued function which is restricted $3$-AP free, and $\E[g] \geq \alpha + \eta$ as required in the statement of the lemma.
  We thus assume that $\E[f^{\sharp}\circ M_{\vec{u},\vec{v}}|_{J\rightarrow z}]\leq \eta$ always. By Markov's inequality we thus get that
  \begin{align*}
  \Prob{z}{\E[f^{\sharp}\circ M_{\vec{u},\vec{v}}|_{J\rightarrow z}]\leq -\eta^{1/4}}
  =\Prob{z}{\eta - \E[f^{\sharp}\circ M_{\vec{u},\vec{v}}|_{J\rightarrow z}]\leq \eta+\eta^{1/4}}
  &\leq \frac{\Expect{z}{\eta - \E[f^{\sharp}\circ M_{\vec{u},\vec{v}}|_{J\rightarrow z}]}}{\eta+\eta^{1/4}}\\
  &=\frac{\eta-\E[f^{\sharp}]}{\eta+\eta^{1/4}}\\
  &\leq \frac{4\sqrt{\eta}}{\eta^{1/4}}\\
  &=4\eta^{1/4}.
  \end{align*}
  Hence, with probability at least $p^{-r}/4 - 4\eta^{1/4}\geq p^{-r}/8$ we have that $z$ satisfies the event $E$,
  the inequality~\eqref{eq:second_item_holds2} holds and $\E[f^{\sharp}\circ M_{\vec{u},\vec{v}}|_{J\rightarrow z}]\geq -\eta^{1/4}$.
  We call such $z$ excellent.

  Let $z$ be excellent, and let $z'\in \mathbb{F}_p^{n'}$.
  If $\E[f^{\sharp}\circ M_{\vec{u},\vec{v}}|_{\substack{J\rightarrow z\\ \overline{I_z}\rightarrow z'_{\overline{I_z}}}}]\geq \eta$
  then we are similarly done, hence we assume that $\E[f^{\sharp}\circ M_{\vec{u},\vec{v}}|_{\substack{J\rightarrow z\\ \overline{I_z}\rightarrow z'_{\overline{I_z}}}}]\leq \eta$.
  By Markov's inequality
  \begin{align*}
  \Prob{z'}{\E[f^{\sharp}\circ M_{\vec{u},\vec{v}}|_{\substack{J\rightarrow z\\ \overline{I_z}\rightarrow z'_{\overline{I_z}}}}]\leq - \eta^{1/8}}
  \leq \frac{\E_{z'}[\eta - \E[f^{\sharp}\circ M_{\vec{u},\vec{v}}|_{\substack{J\rightarrow z\\ \overline{I_z}\rightarrow z'_{\overline{I_z}}}}]]}{\eta+\eta^{1/8}}
  &=\frac{\eta - \E[f^{\sharp}\circ M_{\vec{u},\vec{v}}|_{J\rightarrow z}]}{\eta+\eta^{1/8}}\\
  &\leq
  \frac{\eta + \eta^{1/4}}{\eta+\eta^{1/8}},
  \end{align*}
  which is at most $2\eta^{1/8}$. Thus, with probability at least
  $1-\frac{p^{-5r}}{\sqrt{10r}}-2\eta^{1/8} \geq 1/2$ we have
  that $z'$ satisfies~\eqref{eq:second_item_holds} and
  $\E[f^{\sharp}\circ M_{\vec{u},\vec{v}}|_{\substack{J\rightarrow z\\ \overline{I_z}\rightarrow z'_{\overline{I_z}}}}]\leq - \eta^{1/8}$.

  Overall, we get that with probability at least $\frac{p^{-r}}{8}\cdot \frac{1}{2}$ we have that $z$ is excellent,~\eqref{eq:second_item_holds}
  holds and $\E[f^{\sharp}\circ M_{\vec{u},\vec{v}}|_{\substack{J\rightarrow z\\ \overline{I_z}\rightarrow z'_{\overline{I_z}}}}]\geq - \eta^{1/8}$.
  We fix such $z$ and $z'$. We inspect $P^{\sharp}\circ M_{\vec{u},\vec{v}}|_{\substack{J\rightarrow z\\ \overline{I_z}\rightarrow z'_{\overline{I_z}}}}$ and
  see that it is constant. Indeed, its value on an input $x$ is $P^{\sharp}(w)$ where
  $w=M_{\vec{u},\vec{v}}(x,z',z)$. On coordinates $i\in J$ or $i\not\in I_z$ the value of $w_i$ depends only on
  $z$ and $z'$ by construction and hence is fixed, and on coordinates $i\in I_z$ the value of $w_i$ is $x_j$ where
  $j$ is the unique index such that $i\in R_j$. It follows that
  \[
  P^{\sharp}(w)
  =C(z,z')\prod\limits_{i\in I_z}h(w_i)
  =C(z,z')\prod\limits_{j\text{ good}}\prod\limits_{i\in R_j}h(x_j)
  =C(z,z')\prod\limits_{j\text{ good}}h(x_j)^{\card{R_j}}
  =C(z,z'),
  \]
  where the last transition follows as $\card{R_j} = r$ and by Fact~\ref{fact:trivial}.
  Here, $C(z,z')$ is a complex number of absolute value $1$ depending only on $z,z'$.
  Thus, the fact that~\eqref{eq:second_item_holds} holds means that
  $\card{\E[f^{\sharp}\circ M_{\vec{u},\vec{v}}|_{\substack{J\rightarrow z\\ \overline{I_z}\rightarrow z'_{\overline{I_z}}}}]}\geq \eps'/2$.
  As we know that $\E[f^{\sharp}\circ M_{\vec{u},\vec{v}}|_{\substack{J\rightarrow z\\ \overline{I_z}\rightarrow z'_{\overline{I_z}}}}]\geq -\eta^{1/8}$
  and  $\eta^{1/8} < \eps'/2$ by choice of parameters, it follows that
  $\E[f^{\sharp}\circ M_{\vec{u},\vec{v}}|_{\substack{J\rightarrow z\\ \overline{I_z}\rightarrow z'_{\overline{I_z}}}}] \geq \eps'/2$.
  Thus, $g = \alpha+f^{\sharp}\circ M_{\vec{u},\vec{v}}|_{\substack{J\rightarrow z\\ \overline{I_z}\rightarrow z'_{\overline{I_z}}}}$
  is a Boolean function which is restricted $3$-AP free, and $\E[g]\geq \alpha + \eps'$, and the proof is concluded.
\end{proof}

\subsection{Iterating the Density Increment Argument}
We are now ready to prove Theorem~\ref{thm:main_el}, which clearly follows from Theorme~\ref{thm:main_el_restated} below.
\begin{thm}\label{thm:main_el_restated}
  For all $p$ there is $C = C(p)>0$ such that for all
  $\alpha>0$ there is $N = {\sf exp}({\sf exp}({\sf exp}(C/\alpha^C))$
  such that the following holds for $n\geq N$.
  If $f\colon \mathbb{F}_p^n\to\{0,1\}$ is a function with average $\alpha$, then $f$ contains
  a restricted $3$-AP.
\end{thm}
\begin{proof}
  Assume that this is not the case; then by Lemma~\ref{lem:density_inc_step} we may
  find a function $f_1$ on $n_1\geq \gamma n$ coordinates which also does not contain
  restricted $3$-APs and $\E[f_1]\geq \alpha + \alpha'$, where $\gamma$ and $\alpha'$ are
  from Lemma~\ref{lem:density_inc_step}. We now iterate this, and note after
  at most $1/\alpha'$ times, we get a function $\tilde{f}$ with average at least $0.99$
  on $\tilde{n}\geq \gamma^{1/\alpha'} n$ coordinates which is restricted $3$-AP free.
  As $n\geq N$ we get that $\tilde{n}\geq 10$,
  hence it follows from the union bound that $\tilde{f}$ contains at least $0.97 > 2^{-10}$ of the
  tuples $(x,x+a,x+2a)$ where $x\sim\mathbb{F}_p^{\tilde{n}}$ and $a\sim\{0,1,2\}^{\tilde{n}}$.
  In particular, $\tilde{f}$ is not restricted $3$-AP free, and contradiction.
\end{proof}


\section*{Acknowledgments} 
We thank an anonymous for many helpful comments on an earlier version of this paper.

\bibliographystyle{amsplain}
\newcommand{\etalchar}[1]{$^{#1}$}


\begin{dajauthors}
\begin{authorinfo}[pgom]
  Amey Bhangale\\
  Assistant Professor\\
  University of California Riverside\\
  CA, USA\\
  ameyrbh\imageat{}gmail\imagedot{}com\\
  \url{https://sites.google.com/view/amey-bhangale/home}
\end{authorinfo}
\begin{authorinfo}[laci]
  Subhash Khot\\
  Professor\\
  New York University\\
  New York, USA\\
  khot\imageat{}cims\imagedot{}nyu\imagedot{}edu\\
  \url{https://cs.nyu.edu/~khot/}
\end{authorinfo}
\begin{authorinfo}[andy]
  Dor Minzer\\
  Assistant Professor\\
  Massachusetts Institute of Technology\\
  Cambridge, MA, USA\\
  dminzer\imageat{}mit\imagedot{}edu\\
  \url{https://sites.google.com/view/dorminzer/home}
\end{authorinfo}
\end{dajauthors}

\end{document}